\documentclass[10pt]{IEEEtran}

\usepackage{hyperref}
\usepackage{amsthm}
\usepackage{amsmath}
\usepackage{amssymb}
\usepackage{cite}
\usepackage{subcaption}
\usepackage{mathtools}
\usepackage{amsfonts}

\mathtoolsset{showonlyrefs}

\title{\LARGE \bf Revisiting Normalized Gradient Descent:\\ Fast Evasion of Saddle Points}

\author{\hspace{2pt}Ryan Murray\IEEEauthorrefmark{1}\IEEEauthorrefmark{3}\,, Brian Swenson\IEEEauthorrefmark{1}\IEEEauthorrefmark{2}\,,  and  Soummya Kar\IEEEauthorrefmark{2}}

\newtheorem{theorem}{Theorem}

\newtheorem{corollary}[theorem]{Corollary}

\newtheorem{definition}[theorem]{Definition}

\newtheorem{proposition}[theorem]{Proposition}
\newtheorem{remark}[theorem]{Remark}
\newtheorem{assumption}[theorem]{Assumption}
\newtheorem{example}[theorem]{Example}

   \def\vx{{\bf x}}

  \def\calL{\mathcal{L}}

\newcommand{\e}{\varepsilon}

\newcommand{\ds}{\, ds}

\newcommand{\x}{{ \bf x}}
\newcommand{\R}{\mathbb{R}}

\newcommand{\grad}{\nabla}
\newcommand{\lambdaMin}{|\lambda|_{\textup{min}}}
\newcommand{\lambdaMax}{|\lambda|_{\textup{max}}}
\newcommand{\ddt}{\frac{d}{dt}}
\newcommand{\cl}{\textup{cl}}
\newcommand{\diag}{\textup{diag}}

\begin{document}

\maketitle

\renewcommand{\thefootnote}{\fnsymbol{footnote}}

\footnotetext[1]{These authors contributed equally.}
\footnotetext[2]{Department of Electrical and Computer Engineering,
Carnegie Mellon University, Pittsburgh, PA, USA (brianswe@ece.cmu.edu, soummyak@andrew.cmu.edu).}
\footnotetext[3]{Department of Mathematics, Pennsylvania State University, State College, PA, USA (rwm22@psu.edu).}

\renewcommand{\thefootnote}{\arabic{footnote}}

\thispagestyle{empty}
\begin{abstract}
The note considers \emph{normalized gradient descent} (NGD), a natural modification of classical gradient descent (GD) in optimization problems. A serious shortcoming of GD in non-convex problems is that GD may take arbitrarily long to escape from the neighborhood of a saddle point. This issue can make the convergence of GD arbitrarily slow, particularly in high-dimensional non-convex problems where the relative number of saddle points is often large.
The paper focuses on continuous-time descent. It is shown that, contrary to standard GD, NGD escapes saddle points ``quickly.'' In particular, it is shown that (i) NGD ``almost never'' converges to saddle points and (ii) the time required for NGD to escape from a ball of radius $r$ about a saddle point $x^*$ is at most $5\sqrt{\kappa}r$, where $\kappa$ is the condition number of the Hessian of $f$ at $x^*$. As an application of this result, a global convergence-time bound is established for NGD under mild assumptions.
\end{abstract}

\section{Introduction}
Given a differentiable function $f:\R^d\to\R$, the canonical first-order optimization procedure is the method of gradient descent (GD).
In continuous-time, GD is defined by the differential equation
\begin{equation} \label{eqn:grad-dynamics}
\dot \vx = -\grad f(\vx)
\end{equation}
and in discrete-time, GD is defined by the difference equation
\begin{equation} \label{eq_GD_DE}
x_{n+1} = x_n - \alpha_n\grad f(x_n),
\end{equation}
where $\{\alpha_n\}_{n\geq 1}$ is some step-size sequence. The discrete-time GD process \eqref{eq_GD_DE} is merely a \emph{sample and hold} (or \emph{Euler}) discretization of the differential equation \eqref{eqn:grad-dynamics}, and the properties of solutions of \eqref{eqn:grad-dynamics} and \eqref{eq_GD_DE} are closely related \cite{khalil1996noninear,benaim1996dynamical,stoer2013introduction}.
Owing to their simplicity and ease of implementation, GD and related first-order optimization procedures are popular in practice, particularly in large-scale problems where second-order information such as the Hessian can be costly to compute \cite{boyd2004convex}. When the objective function $f$ is convex, GD can be both practical and effective as an optimization procedure. However, when $f$ is non-convex, GD can perform poorly in practice, even when the goal is merely to find a local minimum.

The underlying issue is the presence of saddle points in non-convex functions; the gradient $\grad f(x)$ vanishes near saddle points, which causes GD to ``stall'' in neighboring regions \cite{dauphin2014identifying} (see also Section \ref{sec:saddle-points-GD}). This both slows the overall convergence rate and makes detection of local minima difficult. The detrimental effects of this issue become particularly severe in high-dimensional problems where the number of saddle points may proliferate. Recent work \cite{dauphin2014identifying} showed that in some high-dimensional problems of interest, the number of saddle points increases exponentially relative to the number of local minima, which can dramatically increase the time required for GD to find even a local minimum.

Since first-order dynamics such as GD tend to be relatively simple to implement in large-scale applications, there has been growing interest in understanding the issue of saddle-point slowdown in non-convex problems and how to overcome it \cite{dauphin2014identifying,ge2015escaping,JordanStableManifold,JordanPNAS,AartiSlowdown,reddi2017saddles}. For example, there has been a surge of recent research on this topic in the machine learning community where large-scale non-convex optimization and first-order methods are of growing importance in many applications \cite{choromanska2015loss,sun2016geometric,dauphin2014identifying,sun2015complete,ge2015escaping}.

One intuitively simple method that has been proposed to mitigate this issue is to consider \emph{normalized} gradient descent (NGD). In continuous time, NGD (originally introduced in \cite{cortes2006}) is defined by the differential equation
\begin{equation}\label{eqn:normalized-dynamics}
\dot\vx = -\frac{\grad f(\vx)}{\|\grad f(\vx)\|}
\end{equation}
and in discrete time, NGD (originally introduced in \cite{Nesterov1984NGD}) is defined by the difference equation
\begin{equation}\label{eq_NGD_DE}
x_{n+1} = x_{n} - \alpha_n\frac{\grad f(x_n)}{\|\grad f(x_n)\|},
\end{equation}
where $\{\alpha_n\}_{n\geq 1}$ is some step-size sequence. As with GD,  discrete-time NGD \eqref{eq_NGD_DE} is merely a sample and hold discretization of its continuous-time counterpart \eqref{eqn:normalized-dynamics}.

The normalized gradient $\frac{\grad f(\vx)}{\|\grad f(\vx)\|}$ preserves the direction of the gradient but ignores magnitude. Because $\frac{\grad f(\vx)}{\|\grad f(\vx)\|}$ does not vanish near saddle points, the intuitive expectation (corroborated by evidence \cite{Levy}) is that NGD should not slow down in the neighborhood of saddle points and should therefore escape ``quickly.''

In this note, our goal is to elucidate the key differences between GD and NGD and, more importantly, give rigorous theoretical justification to the intuition that NGD ``escapes saddle points quickly.'' We will focus, in this work, on continuous-time descent. From the control perspective this may be seen as extending the seminal work of \cite{cortes2006} by characterizing saddle-point behavior of NGD. From the optimization perspective, focusing on continuous-time dynamics allows us to more easily characterize the fundamental properties of NGD using a wealth of available analysis tools and follows in the spirit of recent works studying optimization processes through the lens of differential equations \cite{su2014differential,krichene2015accelerated}.

We have three main results, which we state informally here:

\textbf{Main Result 1} (Theorem \ref{prop:stable-manifold}): Our first main result is to show that NGD can only converge to saddle points from a set of initial conditions with measure zero.
We note that this result implies that, generically, NGD only converges to minima of $f$.\footnote{When we say that a property holds generically for an ODE, we mean that it holds from all initial conditions except, possibly, some set with Lebesgue measure zero.}
However, it provides no guarantees about convergence time or saddle-point escape time. (Indeed, this same result is known to hold for GD, which performs poorly in practice due to saddle-point slowdown.)

This result follows as a relatively straightforward application of the stable-manifold theorem from classical ODE theory (see Proposition \ref{prop:arc-length}, Theorem \ref{prop:stable-manifold} and proofs thereof).

\textbf{Main Result 2} (Theorem \ref{thm:main-thm}): Our second
main result is to show that NGD always escapes from saddle points ``quickly.'' More precisely, we show that the maximum amount of time a trajectory of NGD can spend in a ball of radius $r>0$ about a (non-degenerate) saddle point $x^*$ is $5\sqrt{\kappa}r$, where $\kappa$ is the condition number of the Hessian of $f$ at $x^*$ (see Theorem \ref{thm:main-thm}).\footnote{In Theorem \ref{thm:main-thm} we show a slightly more refined result than this; namely, we show that the time spent in the $r$-ball can be upper bounded by $C\sqrt{\kappa}r$, where $C$ is any constant strictly greater than 4. For clarity of presentation we simply fix the constant to be 5 here. See Remark \ref{remark:constant1} for more details.}

We note that this result is independent of the dimension of the problem.
In contrast to this, the saddle-point escape time of GD (i.e., the maximum amount of time a trajectory of GD may take to leave a ball of radius $r$ about a saddle point) is always \emph{infinite}, independent of the function $f$, the particular saddle point $x^*$, or the dimension of the problem. (See Theorem \ref{thm:main-thm} for a precise definition of saddle-point escape time and Remark \ref{remark:GD-escape-time} for a discussion of GD saddle-point escape time.) This is precisely the issue which causes GD to perform poorly in high-dimensional problems with many saddle points.

While a characterization of saddle-point escape time such as Theorem \ref{thm:main-thm} is essential in understanding how NGD can mitigate the problem of saddle-point slowdown in high dimensional optimization \cite{dauphin2014identifying}, the issue is challenging to study due to the discontinuity in the right-hand side of \eqref{eqn:normalized-dynamics}. In particular, the system is not amenable to classical analytical techniques.
We prove Theorem \ref{thm:main-thm} by studying the rate of ``potential energy dissipation'' (to use an analogy from physics) of NGD near saddle points. The methods used are flexible and can be applied to a variety of discontinuous dynamical systems (see Remark \ref{remark-proof-techniques} and proof of Proposition \ref{prop:time-bound}).

\textbf{Main Result 3} (Corollary \ref{cor-finite-time}):  As our final main result, using the local saddle-point analysis noted above (Theorem \ref{thm:main-thm}) we provide a simple global bound on the convergence time of NGD under mild assumptions on $f$.\footnote{We note that classical first-order methods converge to a local minimum over an infinite time horizon. Continuous-time NGD, on the other hand, converges in finite time (cf. \cite{cortes2006}). Hence the bound we provide concerns the convergence time rather than the convergence rate.}

\textbf{Literature Review}: Continuous-time NGD dynamics were
first introduced by Cortes \cite{cortes2006} in the context of distributed multi-agent coordination. In \cite{cortes2006} it was shown that NGD converges to critical points of $f$ in finite time and this result was used to develop distributed gradient coordination algorithms that achieve a desired task in finite time. Our results differ from \cite{cortes2006} primarily in that we characterize the saddle-point behavior of NGD, including demonstrating non-convergence to saddle points and providing a strong characterization of saddle-point escape time.
Furthermore, our results differ from \cite{cortes2006} in that (i) our results show that NGD almost always converges to local minima rather than just the set of critical points of $f$ and
(ii) \cite{cortes2006} considered only local bounds on the convergence time of NGD to local minima. Because we characterize the saddle point behavior of NGD, our results enable \emph{global} bounds on the convergence time of NGD to minima of non-convex functions (see Corollary \ref{cor-finite-time}).

Discrete-time NGD was first introduced by Nesterov \cite{Nesterov1984NGD}
and variants have received increasing attention in the optimization and machine learning communities \cite{kiwiel2001convergence,konnov2003convergence,hazan2015beyond,Levy}.
The problem of coping with saddle points in non-convex optimization has received significant recent attention (see \cite{dauphin2014identifying,ge2015escaping,JordanStableManifold,JordanPNAS,AartiSlowdown,reddi2017saddles} and references therein). Of particular relevance to the present work are results dealing with first-order methods. Recent work along these lines includes the following. The work \cite{JordanStableManifold} shows that the classical stable manifold theorem implies that gradient descent only converges to minima. The work \cite{AartiSlowdown} shows that, even with random initialization, discrete-time
GD can take exponential time to escape saddle points. The work \cite{ge2015escaping} showed that noisy discrete-time GD converges to a local minimum in a polynomial number of iterations.
Our work differs from \cite{ge2015escaping} primarily in that we investigate the role of normalization of the dynamics (rather than noise injection) as a means of accelerating escape from saddle points.

The use of normalization in GD has also been studied in \cite{Levy} where it was shown that discrete-time NGD with noise injection can outperform GD with noise injection \cite{ge2015escaping} in terms of dimensional dependence and the number of iterations required to reach the basin of a local minimum. Numerical simulations of discrete-time noisy NGD and comparisons with discrete-time noisy GD in several problems of interest were also performed in \cite{Levy}. Our work differs from \cite{Levy} in that we study the continuous-time deterministic NGD dynamics \eqref{eqn:normalized-dynamics} (which may be viewed as the mean dynamics of the noise-injected discrete-time NGD \cite{Levy} as the step size is brought to zero), we characterize the stable-manifold for these dynamics near saddle points, and we explicitly characterize the saddle-point escape time.

The work \cite{jin2017escape} improved on the dimensional dependence of the results of \cite{ge2015escaping} and \cite{Levy}, showing that GD with noise injection can reach the basin of a local minimum in a number of iterations with only polylog dependence on dimension.
Our work differs from \cite{jin2017escape} in that we again study the underlying continuous dynamics and perform an explicit local analysis of the dynamics near saddle points. We demonstrate that the local saddle point escape time of NGD can be bounded independent of dimension (Theorem \ref{thm:main-thm}). Moreover, because we show that NGD is a path-length reparametrization of GD, our results also have implications for classical GD. In particular, Theorem \ref{thm:main-thm} together with Proposition \ref{prop:arc-length} shows that a classical GD trajectory can have at most length $5\sqrt{\kappa} r$ (where $\kappa$ is the condition number of the Hessian of $f$ at $x^*$) before it must exit a ball of radius $r$ about a saddle point.


\textbf{Organization}: Section \ref{sec:notation} sets up notation. Section \ref{sec:examples} presents a simple example illustrating the salient features of GD and NGD near saddle points. Section \ref{sec:struct_properties} studies the structural relationship between GD and NGD and presents Theorem \ref{prop:stable-manifold} which shows generic non-convergence to saddle points.  Section \ref{sec:main_result} presents Theorem \ref{thm:main-thm} which gives the saddle-point escape-time bound for NGD. Section \ref{sec:global-bound} presents a simple global convergence-time bound for NGD (Corollary \ref{cor-finite-time}). The proofs of all results are deferred to Section \ref{sec:proofs}.

\section{Preliminaries} \label{sec:notation}
Suppose $f:\R^d\to \R$ is a twice differentiable function. We use the following notation.
\begin{itemize}
\item $\grad f(x)$ denotes the gradient of $f$ at $x$
\item $D^2f(x)$ denotes the Hessian of $f$ at $x$
\item Given a set $S\subset \R^d$, the closure of $S$ is given by $\cl(S)$ and the boundary of $S$ is given by $\partial S$
\item $\calL^d$, $d\geq 1$ denotes the $d$-dimensional Lebesgue measure
\item $B_r(x)$ denotes the ball of radius $r$ about $x\in \R^d$
\item $\| \cdot \|$ denotes the Euclidean norm
\item $d(\cdot,\cdot)$ denotes Euclidean distance
\item $\dot \vx$ is shorthand for $\ddt \vx(t)$
\item Given $C>0$, $|D^3 f(x)|<C$ means that $\vert\frac{\partial ^3 f(x)}{\partial x_i \partial x_j \partial x_k}\vert < C$, $i,j,k=1,\ldots,d$
\item For $A\in \R^{n\times n}$, $\sigma(A)$ denotes the spectrum of $A$
\item $\lambdaMin(A):= \min\{|\lambda|:\lambda\in\sigma(A)\}$
\item  $\lambdaMax(A):= \max\{|\lambda|:\lambda\in\sigma(A)\}$
\item The \emph{condition number} of $A$ is given by $\frac{\lambdaMax(A)}{\lambdaMin(A)}$
\item $\diag(\lambda_1,\ldots,\lambda_d)$ gives a $d\times d$ matrix with $\lambda_1,\ldots,\lambda_d$ on the diagonal
\end{itemize}


We say that a saddle point $x^*$ of $f$ is \emph{non-degenerate} if $D^2f(x^*)$ is non-singular.

For $k\geq 1$, let $C^k$ denote the set of all functions from $\R^d$ to $\R$ that are $k$-times continuously differentiable.
Unless otherwise specified, we will assume the following throughout the paper.
\begin{assumption} \label{a:twice-differentiable}
The objective function $f$ is of class $C^2$.
\end{assumption}

We say that a continuous mapping $\vx:I \to R^d$, over some interval $I =[0,T)$, $0 < T \leq \infty$, is a solution to an ODE with initial condition $x_0$ if $\vx \in C^1$, $\vx$ satisfies the ODE for all $t\in I$, and $\vx(0) = x_0$.


Under assumption \ref{a:twice-differentiable}, there exists a unique solution to \eqref{eqn:grad-dynamics} which exists on the interval $I = \R$ for every initial condition.
A solution $\vx$ to \eqref{eqn:normalized-dynamics} with initial condition $x_0$ satisfying $\nabla f(x_0) \not = 0$, will have a unique solution on some \emph{maximal interval of existence} $[0,T)$, where $T$ is dependent on $x_0$ (see \cite{Perko_ODE} for a formal definition of the maximal interval of existence). Practically, for solutions of \eqref{eqn:normalized-dynamics} the maximal interval of existence is the maximal time interval for which a solution $\vx$ does not intersect with a critical point of $f$.
When we refer to a solution of \eqref{eqn:normalized-dynamics} we mean the solution defined over its maximal interval of existence.

\begin{remark}[Fillipov solutions]
We note that one can handle the discontinuity in the right hand side of \eqref{eqn:normalized-dynamics} by considering solutions of the associated Fillipov differential inclusion \cite{Filippov,cortes2006}.
In order to keep the presentation simple and broadly accessible we have elected to avoid this approach and instead consider solutions only on intervals on which they are classically defined. Practically, the main differences between the two approaches are that (i) solutions in the classical sense cease to exist when they reach a saddle point or local minimum whereas Fillipov solutions remain well defined at these points, and (ii) Fillipov solutions may not be differentiable at times when solutions reach or depart from critical points. In particular, Fillipov solutions to \eqref{eqn:normalized-dynamics} may sojourn indefinitely at saddle points (and local maxima) of $f$ and remain at non-degenerate minima of $f$ once reached.
Our results and analysis extend readily to solutions in this sense modulo minor technical modifications.
\end{remark}

%

The following two definitions are standard from classical ODE theory.
\begin{definition}[Orbit of an ODE] \label{def:orbit}
  Let $x(t)$ be the solution of some ODE on the interval $[0,T)$. Assume that $x(0) = x_0$ and that $[0,T)$ is the maximal interval on which $x(t)$ is the unique solution of the ODE with initial value $x_0$ (here $T=\infty$ is permitted). Then the \emph{orbit} corresponding to the initial condition $x_0$ is defined to be the set $\gamma_{x_0}^+ := \{x\in \R^d:~ \vx(t)= x \mbox{ for some } t \in [0,T)\}$.
\end{definition}
We note that $\gamma_{x_0}^+$ in the above definition is often referred to as a \emph{forward orbit}; to simplify nomenclature, we will refer to it simply as an orbit.

Given a differentiable curve $\vx: [0,T)\to \R^d$, the \emph{arc length} of $\vx$ at time $t< T$ is given by $L(t):= \int_{0}^t |\dot \vx(s)|\ds$, and we let $L(T) := \lim_{t\to T} L(t)$.
\begin{definition} [Arc-Length Reparametrization] \label{def:arc-reparam}
  Suppose $\vx:[0,T) \to\R^d$ is a differentiable curve in $\R^d$ with arc length at time $t$ given by $L(t)$. We say that $\tilde \vx: I \to \R^d$, $I=[0,L(T))$ is an \emph{arc-length reparametrization} of $\vx(t)$ if there holds $\vx(t) = \tilde \vx(L(t))$ for all $t\in [0,T)$.
\end{definition}

We say that a property holds for \emph{almost every} element in a set $A\subseteq \R^d$, $d\geq 1$, if the subset of $A$ where the property fails to hold has $\calL^d$-measure zero. Likewise, we say that a property holds for almost every solution of an ODE if the property holds for solutions starting from almost every initial condition.


\section{Saddle-Point Behavior of GD and NGD: Examples and Intuition} \label{sec:examples}

\subsection{Saddle Points and GD} \label{sec:saddle-points-GD}
The following simple example illustrates the behavior of GD near saddle points.
\begin{example} \label{example-GD}
Suppose the objective function is given by
\begin{equation} \label{eqn:example-f}
  f(x) = \frac{1}{2} x^T A x, \quad\quad A =
\begin{pmatrix}
 1 & 0\\
 0 & -1
\end{pmatrix}
\end{equation}
and note that the origin is a saddle point of $f$. The associated GD dynamics \eqref{eqn:grad-dynamics} reduce to a simple linear system of the form
\begin{equation} \label{eqn:lin-system}
\ddt \vx(t) = -A\vx(t)
\end{equation}
with solution $\vx(t) = e^{-At}x_0$, for initial condition $\vx(0) = x_0 \in \R^2$.

By classical linear systems theory we see that solutions of this system will only converge to the origin if they start with initial conditions in the stable eigenspace of $-A$, which is given by $E_s := \{x=(x^1,x^2)\in\R^2: x^2 = 0\}$). Note that this is a set of initial conditions with Lebesgue measure zero.

Let $r>0$ and consider the following question: What is the maximum amount of time that a solution of \eqref{eqn:lin-system} may spend in a ball of radius $r>0$ about the origin?
It is straightforward to verify that trajectories not converging to $0$ may take arbitrarily long to leave $B_r(0)$, and so the time it could potentially take to escape saddle points is unbounded. Indeed, note that for $\e \in (0,r)$, a trajectory of \eqref{eqn:lin-system} starting on $\partial B_r(0)$ which enters $B_\epsilon(0)$ must spend at least time $-r\log(\e)$ inside the $r$-ball before it may enter the $\e$ ball.
\end{example}

These same basic properties generalize to GD in higher dimensional systems: Solutions of GD may only converge to a saddle point from a set of initial conditions with measure zero, but the time required to escape neighborhoods of the saddle is always infinite. This is made precise in the following remark.
\begin{remark}[Saddle-Point Escape Time of GD] \label{remark:GD-escape-time}
Informally, given a function $f$, a saddle point $x^*$ of $f$, and an $r>0$
we refer to the ``saddle-point escape time'' of an optimization process as the maximum amount of time a trajectory which does not converge to $x^*$ may spend in a ball of radius $r$ about $x^*$. In GD, the saddle point escape time is always infinite. That is, for arbitrary objective function $f$, saddle point $x^*$, and radius $r>0$ there holds
  \begin{equation}
    \sup_{\substack{x_0\in \partial B_r(x^*)\\ x^* \notin \cl(\gamma_{x_0}^+)}}  \calL^1\bigg( \Big\{ t\in [0,\infty):~ \vx_{x_0}(t)\in B_r(x^*) \Big\} \bigg) = \infty,
\end{equation}
where $\vx_{x_0}$ is the solution of \eqref{eqn:grad-dynamics} with initial condition $x_0$.
\end{remark}
This is precisely the issue which causes GD to perform poorly in high-dimensional problems with many saddle points. In this paper we will see that NGD significantly mitigates this issue---rather than having an infinite saddle-point escape time, the saddle point escape time of NGD is at most $5\sqrt{\kappa}r$, where $\kappa$ is the condition number of $D^2 f(x^*)$.

\subsection{Saddle Points and NGD} \label{sec:examples-NGD}
We will now consider the behavior of NGD near the saddle point in the above example.

In order to better understand this issue, it is helpful to characterize the relationship between GD and NGD. In Section \ref{sec:struct_properties} we will see that GD and NGD are closely linked---the two systems are ``topologically equivalent'' \cite{Perko_ODE} and solutions of NGD are merely arc-length reparameterizations of GD solutions (see Definition \ref{def:arc-reparam}). In practical terms this means that if one considers orbits of NGD and GD starting from the same initial condition $x_0\in\R^d$, the orbits generated by the two systems are identical (see Definition \ref{def:orbit}). The solutions of each system only vary in how quickly they move along the common orbit. In particular, since NGD always ``moves with speed 1'' (i.e., $\|\dot \vx(t)\| = 1, ~\forall t\geq 0$) the length of an arc generated by NGD up to time to time $t$ is precisely $t$ (this is what it means to be an arc-length reparameterization). As an important result of this characterization, we will see that NGD ``almost never'' converges to saddle points (see Theorem \ref{prop:stable-manifold}).

While a solution of GD may move arbitrarily slowly as it passes near a saddle point, a solution of NGD starting at the same initial condition will move along the same orbit with constant speed, not slowing near the saddle point. This is illustrated in Fig. \ref{fig:NGD_v_GD}.

Consider NGD with $f$ as defined in Example \ref{example-GD} (see \eqref{eqn:example-f}). Given the simple linear structure of the corresponding GD ODE \eqref{eqn:lin-system} it is straightforward to verify that the arc-length of any trajectory of GD (or equivalently NGD) intersecting $B_r(0)$ is upper bounded by $2r$ and hence the maximum time a trajectory of NGD may spend in $B_r(0)$ is $2r$ (see Fig. \ref{fig:NGD_v_GD}).

This simple example may be generalized to higher dimensions. Let $f:\R^d\to \R$, $d\geq 2$ be given by $f(x) = x^T Ax$, with $A = \diag(\lambda_1,\ldots,\lambda_d)$ with $|\lambda_i|=1$ for all $i=1,\ldots,d$, and at least one $\lambda_i>0$ and one $\lambda_i<0$. Given the simple structure of the corresponding GD ODE $\dot \vx = -A\vx$, it is straightforward to show that the arc-length of any trajectory of GD intersecting $B_r(0)$ (and hence the amount of time spent by NGD in $B_r(0)$) is upper bounded by $2r$, independent of the dimension $d$.

Note that in this example, the condition number of $D^2 f(0)$ is 1. In general, as the condition number increases, the time spent by NGD in $B_r(0)$ may increase. Theorem \ref{thm:main-thm} captures this relationship for general $f$ (satisfying Assumption \ref{a:twice-differentiable}).


\begin{remark}
We note that the bound that will be established in Theorem \ref{thm:main-thm} is conservative. In particular, suppose $f:\R^d\to\R$ is quadratic of the form $f(x) = x^TAx$, with $A\in\R^{d\times d}$ diagonal and non-singular. Then one can show that time spent by a trajectory of NGD in $B_r(0)$ is at most $2\sqrt{d}r$. This bound holds even as the condition number of $D^2 f(0)$ is brought to $\infty$.\footnote{This is shown by bounding the arc length of the corresponding linear GD ODE $\dot \vx = -A\vx$. Intuitively, if $A$ is well conditioned, then trajectories of the ODE passing near 0 travel along a ``direct route'' to and away from 0. If $A$ is ill conditioned, then trajectories of the ODE travel a ``Manhattan route'' to and away from 0, with movement tangential to the stable eigenspace of $A$ occurring along only one stable eigenvector at a time.} Thus, while an ill-conditioned saddle point can slow the escape time of NGD, this example suggests that in the worst case as the condition number is brought to $\infty$, the time spent by NGD in $B_r(x^*)$ about a saddle point $x^*$ can be bounded by $C\sqrt{d}r$, where $C>0$ is some universal constant independent of dimension and condition number. An in-depth investigation of this issue is outside the scope of this note.
\end{remark}

%

\begin{figure}
    \centering
    \begin{subfigure}[b]{0.2\textwidth}
        \includegraphics[width=\textwidth]{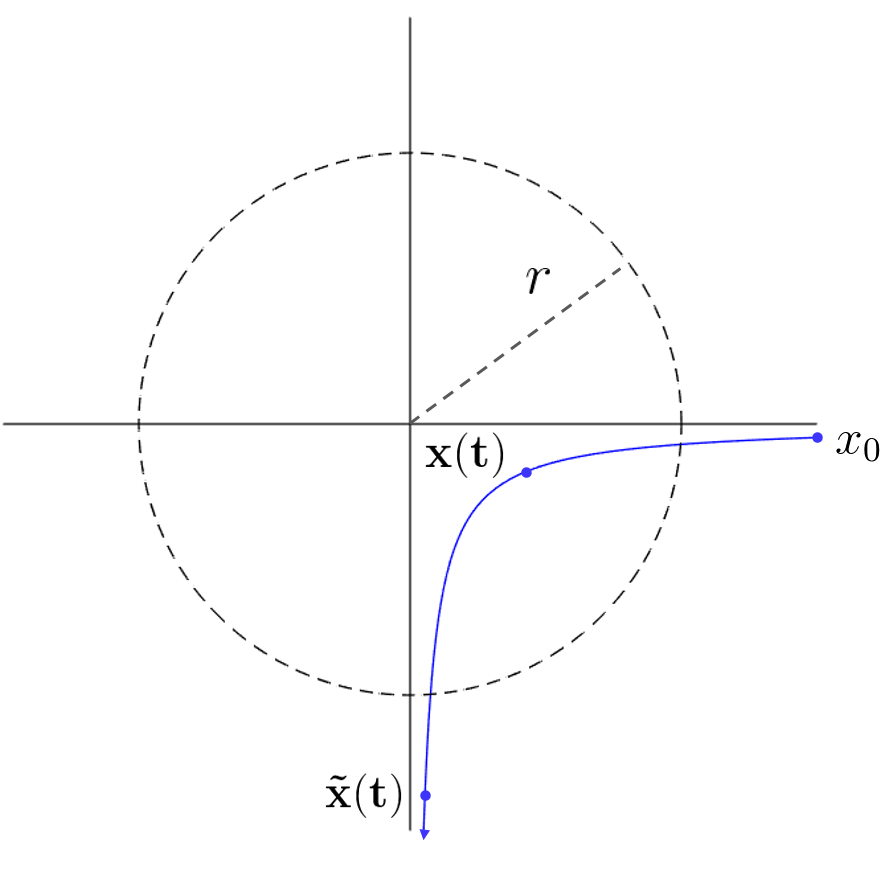}
        \caption{ }
        \label{fig:gull}
    \end{subfigure}
    ~~~ 
    \begin{subfigure}[b]{0.19\textwidth}
        \includegraphics[width=\textwidth]{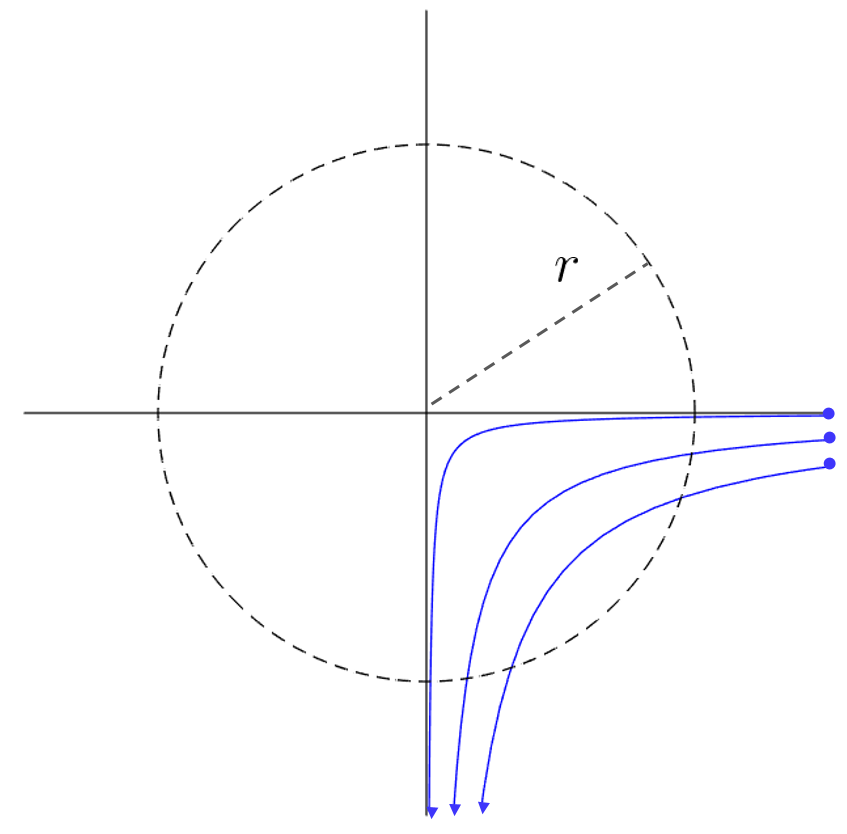}
        \caption{}
        \label{fig:mouse}
    \end{subfigure}
    \caption{\small (a) Common orbit shared by the solutions of GD and NGD starting at the same initial condition $x_0$ with the objective function given by \eqref{eqn:example-f}. At time $t$, trajectory of GD (given by x(t)) stalls near the saddle point while trajectory of NGD (given by ~x(t)) moves along the same orbit with constant speed without slowing down near the saddle point. (b) As $x_0$ approaches the stable eigenspace (the horizontal axis) the length of the orbit inside the ball approaches $2r$.}\label{fig:NGD_v_GD}
\end{figure}

\section{NGD: Structural Properties and Generic Convergence to Local Minima} \label{sec:struct_properties}
The following proposition establishes the basic structural relationship between GD and NGD.
\begin{proposition}\label{prop:arc-length}
 Let $\vx(t)$ and $\tilde \vx(t)$ be solutions of \eqref{eqn:grad-dynamics} and \eqref{eqn:normalized-dynamics} respectively, with the same initial condition $x_0$, over maximal intervals $[0,T)$ and $[0,\tilde T)$ respectively. Then $\tilde \vx(t)$ is an arc length reparametrization of $\vx(t)$, and $\tilde \vx(t) = \vx(s(t))$ for some strictly increasing function $s(t)$, with $s(0) = 0$ and $s(\tilde{T}) = T$.
\end{proposition}

This result means that (classical) solutions of \eqref{eqn:grad-dynamics} and \eqref{eqn:normalized-dynamics} starting at the same initial condition have identical orbits (see Definition \ref{def:orbit}); the solutions only differ in the speed with which they move along the common orbit.\footnote{In other words, the dynamical systems defined by  \eqref{eqn:grad-dynamics} and \eqref{eqn:normalized-dynamics} are topologically equivalent \cite{Perko_ODE} with the concomitant homeomorphism given by the identity.}

The following result shows that NGD may only converge to non-degenerate saddle points from a measure-zero set of initial conditions. Part (i) of the proposition considers a slightly weaker condition than non-degeneracy as discussed in Section \ref{sec:notation}. In particular, we will require that at least one eigenvalue of $D^2 f(x^*)$ be negative. Saddle points satisfying this condition are sometimes referred to in the literature as \emph{rideable} or \emph{strict} saddle points \cite{ge2015escaping,sun2015nonconvex}.

%

\begin{theorem} [Non-Convergence to Saddle Points] \label{prop:stable-manifold} $~$\\
  \noindent (i) Let $x^*$ be a saddle point of $f$ such that there exists a $\lambda \in \sigma(D^2 f(x^*))$ with $\lambda < 0$. Then solutions to \eqref{eqn:normalized-dynamics} can only reach or converge to $x^*$ from a set of initial conditions with Lebesgue measure zero.\\
  \noindent (ii) Suppose that each saddle point of $f$ is non-degenerate. Then the set of initial conditions from which solutions to \eqref{eqn:normalized-dynamics} reach or converge to a saddle point has Lebesgue measure zero.
\end{theorem}

Since a non-degenerate saddle point $x^*$ necessarily has at least one strictly negative eigenvalue in $\sigma(D^2 f(x^*))$, Theorem \ref{prop:stable-manifold} immediately implies that solutions to \eqref{eqn:normalized-dynamics} may only converge to non-degenerate saddle points from a measure zero set of initial conditions.

\begin{remark} [Uniqueness of Fillipov Solutions]
While we do not deal explicitly with Fillipov solutions to \eqref{eqn:normalized-dynamics} in this note, we note that Fillipov solutions of \eqref{eqn:normalized-dynamics} are classical so long as they do not intersect with critical points of $f$. In particular, for functions in which all critical points are non-degenerate, Fillipov solutions are classical until they intersect with critical points, and unique so long as they do not intersect with saddle points or maxima.
Theorem \ref{prop:stable-manifold} shows that Fillipov solutions are unique from almost all initial conditions in functions with non-degenerate critical points.
\end{remark}

It follows from Propositions \ref{prop:arc-length} and \ref{prop:stable-manifold} that solutions of NGD exist and are unique for almost every initial condition. We note that both of these results follow as elementary applications of classical ODE theory (See Section \ref{sec:proofs}).

We also note that this issue (generic non-convergence to saddle points, as in Proposition \ref{prop:stable-manifold}) was considered for discrete-time GD \eqref{eq_GD_DE} in the recent work \cite{JordanStableManifold}. Addressing the question of ``stable manifold'' theorems for the discrete analog of \eqref{eqn:normalized-dynamics} will be a subject of future work.

\section{Fast Escape From Saddle Points} \label{sec:main_result}
The following theorem gives our main result regarding fast escape from saddle points. The theorem provides a simple estimate on the amount of time that trajectories of NGD can spend near saddle points.
\begin{theorem} [Saddle-Point Escape Time] \label{thm:main-thm}
 Let $C>4$ and suppose $x^*$ is a non-degenerate saddle point of $f$. Then for all $r>0$ sufficiently small, any trajectory of \eqref{eqn:normalized-dynamics}
 that does not reach or converge to $x^*$
 can spend time at most time $C\sqrt{\kappa}r$ in the ball $B_r(x^*)$, where $\kappa$ is the condition number of $D^2 f(x^*)$. That is, if $\vx_{x_0}$ is a solution to \eqref{eqn:normalized-dynamics} with initial condition $x_0$ and maximal interval of existence $[0,T_{x_0})$, $T_{x_0}\leq \infty$, and $x^* \notin \cl(\gamma_{x_0}^+)$, then
 $$
 \calL^1\bigg(\Big\{t \in [0,T_{x_0}): \vx_{x_0}(t) \in B_r(x^*) \Big\} \bigg) \leq C\sqrt{\kappa}r.
 $$
\end{theorem}

We recall that by Theorem \ref{prop:stable-manifold}, solutions of \eqref{eqn:normalized-dynamics} can only reach or converge to saddle points from a set of initial conditions with measure zero, hence the theorem hold for solutions starting from almost every initial condition.


In order to underscore the significance of this result, we recall that the saddle-point escape time of GD (i.e., the time required to escape a ball of radius $r>0$ about a saddle point) is infinite, independent of $f$, $d$, $x^*$, and $r$ (see Remark \ref{remark:GD-escape-time}), which causes GD to perform poorly in problems with many saddle points. In contrast to this, Theorem \ref{thm:main-thm} shows that trajectories of NGD always escape a ball of radius $r$ within time $5\sqrt{\kappa}r$.\footnote{As in the introduction, to emphasize the key features of this result we fix the constant $C$ to be 5 here. Of course, the theorem holds for the constant $C$ fixed to any value strictly greater than 4. See Remark \ref{remark:constant1} for more details.}

Furthermore, we recall that Proposition \ref{prop:arc-length} showed that orbits of GD and NGD coincide. Thus, away from saddle points (where GD is generally ``well behaved'') GD and NGD behave in an essentially identical manner in that they follow identical trajectories and the velocity of each can be bounded from below.

%


A few remarks are now in order.
\begin{remark}[Values of constant $C$] \label{remark:constant1}
The above theorem holds with the constant $C$ set to any value strictly greater than 4. The proof of the estimate in the theorem utilizes several Taylor series approximations. There is a tradeoff inherent in this proof technique---as $C$ approaches 4, the range of permissible values of $r>0$ where the Taylor approximation (and hence, the theorem) is applicable shrinks to zero. For clarity of presentation and to emphasize the key features of this result we find it convenient to simply fix the constant to be 5 in the abstract and introduction. See Proposition \ref{prop:time-bound} and proof thereof for more details.
\end{remark}
\begin{remark}[Permissible values of $r$] \label{remark:values-of-r}
The range of values of $r>0$ where Theorem \ref{thm:main-thm} holds depends both on the constant $C$ and the magnitude of higher order derivatives near the saddle point $x^*$. In particular, the result holds so long as the Taylor estimates \eqref{inequality2}, \eqref{eqn:taylor-estimate2} used in the proof are valid. If one assumes that $f$ is more than twice differentiable and assumes bounds on the magnitude of the higher order derivatives near $x^*$, then the radius where these estimates hold can be bounded, and a more precise statement can be made about the permissible values of $r$ in Theorem \ref{thm:main-thm}. For example, if one assumes that $|D^3 f(x)| < \hat C$ is uniformly bounded for some $\hat C>0$
then Theorem \ref{thm:main-thm} holds for all $r\in (0,\bar r)$, where $\bar r = 6\kappa^{-1/2}\hat C^{-1}\lambdaMax(D^2 f(x^*))\left(\frac{C(3\kappa+2)}{6C\kappa+16} - \frac{1}{2} \right)$, and where $\kappa$ is the condition number of $D^2 f(x^*)$. This is verified by confirming that the Taylor estimates \eqref{inequality2}, \eqref{eqn:taylor-estimate2} used in the proof of Proposition \ref{prop:time-bound} are valid in the ball $B_{\hat r}(0)$, $\hat r = \kappa^{1/2} r$, for values of $r$ in this range.
\end{remark}
\begin{remark}[Non-Applicability of the Hartman-Grobman Theorem]\label{remark:hart-grob}
The Hartman-Grobman theorem from classical differential equations states that near non-degenerate saddle points one can construct a homeomorphism mapping the trajectories of a non-linear ODE to trajectories of the associated linearized system \cite{Perko_ODE}.
It is simple to show that Theorem \ref{thm:main-thm} holds when $f$ is quadratic (and hence the associated NGD system is topologically identical to a linear system); see Section \ref{sec:saddle-points-GD}. Thus, one might expect Theorem \ref{thm:main-thm} to hold for general (non-quadratic) $f$ by the Hartman-Grobman theorem.
However, the homeomorphisms constructed in the Hartman-Grobman theorem are in general not smooth, and so will not preserve trajectory length, and cannot be used to prove a bound such as Theorem \ref{thm:main-thm}. Instead one must resort to more analytical techniques to study path length; see the proof of Proposition \ref{prop:time-bound} below.
\end{remark}
\begin{remark}[Theorem \ref{thm:main-thm} Proof Technique] \label{remark-proof-techniques}
Here, the key idea of the proof of Theorem \ref{thm:main-thm} relies on establishing a differential inequality between the ``potential'' $f$ and the ``potential dissipation rate'' $\frac{d}{dt} f(\x(t))$. The methods are flexible, and may be applicable to other non-smooth settings. In a previous work \cite{swenson2017fictitious} the authors utilized similar techniques to study non-smooth dynamics in game-theoretical problems.
\end{remark}

\section{A Global Convergence-Time Bound}\label{sec:global-bound}
We will now use the above results to prove a simple corollary bounding the maximum amount of time that trajectories can take to reach local minima under \eqref{eqn:normalized-dynamics}.


We will make the following assumptions.
\begin{assumption}
\label{a:bdd-3rd-derivative}
The function $f$ is of class $C^3$ and $|D^3 f(x)| \leq \hat C$ uniformly for all $x\in \R^d$, for some $\hat C>0$.
\end{assumption}
This assumption ensures that there exists a single $r>0$ such that Proposition \ref{prop:time-bound} holds within a ball of radius $r$ about \emph{every} critical point (see Remark \ref{remark:values-of-r}).

Next we assume a uniform bound on the magnitude of eigenvalues of the Hessian at critical points.
\begin{assumption}
\label{a:uniform-eig-bd}
There exist constants $\lambdaMax, \lambdaMin >0$
such that for every critical point $x^*$ of $f$ there holds $\lambdaMin \leq |\lambda| \leq \lambdaMax$ for all $\lambda \in \sigma(D^2 f(x^*))$.
\end{assumption}

The next assumption ensures that at any point $x\in \R^d$, either the gradient of $f$ at $x$ is large (guaranteeing fast local improvement of descent techniques), or $x$ is close to a critical point.
\begin{assumption}
\label{a:strict-saddle}
Fix $C>4$. Assuming Assumptions \ref{a:bdd-3rd-derivative} and \ref{a:uniform-eig-bd} hold, let $r>0$ be chosen so that Theorem \ref{thm:main-thm} (or equivalently,  Proposition \ref{prop:time-bound}) holds with constant $C$ at every critical point and so that $r \leq \frac{\lambdaMin}{\hat C}$.
Furthermore, assume that there exists a constant $\nu>0$ such that for all $x\in \R^d$ either
\begin{displaymath}
    \|x-x^*\| < r, \text{ with } \nabla f(x^*) =0  \textbf{ or } \|\nabla f(x)\| > \nu.
\end{displaymath}
\end{assumption}
Assumptions \ref{a:uniform-eig-bd} and \ref{a:strict-saddle} together are similar to the strict saddle property assumed in \cite{ge2015escaping},\cite{Levy}. The main difference is that here we assume a uniform (lower) bound on the minimum-magnitude eigenvalue of the Hessian at all critical points rather than just saddle points, and we assume a uniform (upper) bound on the maximum-magnitude eigenvalue of the Hessian at all critical points.
The final assumption ensures that a descent process will eventually converge to some point rather than expanding out infinitely. This assumption is naturally satisfied, for example, if $f$ is coercive (i.e., $f(x)\to \infty$ as $\|x\|\to\infty$).
\begin{assumption}
\label{a:invariant-set}
There exists an $R>0$ such that trajectories of \eqref{eqn:normalized-dynamics} that begin in $B_R(0)$ remain in $B_R(0)$ for all $t\geq 0$.
\end{assumption}

Let $R>0$ be as in Assumption \ref{a:invariant-set} and let
\begin{equation}\label{def-M}
M:= \sup_{x\in B_R(0)} |f(x)|.
\end{equation}
Note that since $f$ is continuous, $M<\infty$.



The following result gives a simple estimate on the amount of time the dynamics \eqref{eqn:normalized-dynamics} will take to reach a local minimum.
\begin{corollary}\label{cor-finite-time}
  Suppose that every saddle point $x^*$ of $f$ is non-degenerate and that Assumptions \ref{a:bdd-3rd-derivative}--\ref{a:invariant-set} hold. Then for almost every initial condition inside $B_R(0)$, solutions of \eqref{eqn:normalized-dynamics} will converge to a local minimum in at most time $2M\nu^{-1} + C\sqrt{\frac{\lambdaMax}{\lambdaMin}}\frac{(R+r)^d}{r^{d-1}}$, where $C>4$ is the constant in Assumption \ref{a:strict-saddle}.
\end{corollary}


\section{Proofs of Main Results} \label{sec:proofs}
We now present the proofs of the results found in Sections \ref{sec:struct_properties} -- \ref{sec:main_result}.

We begin by presenting the proofs of Propositions \ref{prop:arc-length} and Theorem \ref{prop:stable-manifold}, which follow from elementary applications of classical ODE theory.

\begin{proof}[Proof of Proposition \ref{prop:arc-length}]
Given a solution $\vx$ to \eqref{eqn:grad-dynamics}, one can reparametrize the trajectory by arc length, i.e., $\hat \x (t) =  \x(L(t))$, and $\|\ddt \hat \x(t)\| = 1$. Using the chain rule we find that $\frac{d}{dt} \hat \x(t) = -\frac{\nabla f(\hat \x(t))}{\|\nabla f(\hat \x(t))\|}$. Since the solutions are classical, uniqueness of solutions for ODE gives us that $\tilde \x$ and $\hat \x$ must be equal.
\end{proof}

\begin{proof}[Proof of Theorem \ref{prop:stable-manifold}]
We begin by proving part (i) of the theorem. Solutions to \eqref{eqn:grad-dynamics} which converge to such a saddle point are contained within a \emph{stable manifold}, i.e. a smooth surface of at most dimension $n-1$. Such a surface will be a set with Lebesgue measure zero. Proof and details of such a result may be found in \cite{Perko_ODE}. The result then follows from Proposition \ref{prop:arc-length}.

Part (ii) of the theorem follows from the fact that if all saddle points are non-degenerate, then all saddle points are isolated. Hence, the set of saddle points is countable. By part (i) of the theorem, the union of the stable manifolds for all saddle points is a set with Lebesgue measure zero.
\end{proof}

The following proposition proves Theorem \ref{thm:main-thm}.  The proposition is stated in slightly more general terms than Theorem \ref{thm:main-thm} in order to account for the behavior of NGD near minima as well as saddle points.

\begin{proposition}\label{prop:time-bound}
Let $C>4$, let $x^*\in\R^d$ be a non-degenerate critical point of $f$, and let $\vx(t)$ be a solution of \eqref{eqn:normalized-dynamics} with arbitrary initial condition $x_0\not = x^*$ and maximal interval of existence $[0,T_{x_0})$.
For all $r>0$ sufficiently small, the time spent by $\vx(t)$ in $B_r(x^*)\backslash \{x^*\}$ is bounded according to
$$
\calL^1\big(\big\{t\geq [0,T_{x_0}): \vx(t) \in B_r(x^*)\backslash \{x^*\} \big\}\big) \leq C\sqrt{\kappa} r,
$$
where $\kappa = \frac{\lambdaMax(D^2 f(x^*))}{\lambdaMin(D^2 f(x^*))}$.
\end{proposition}

\begin{proof}
Without loss of generality, assume $x^*=0$ and let $H := D^2 f(0)$.
For $x\in \R^d$ define $\tilde d(x) := \sqrt{x^T |H| x}$, where $|B| := \sqrt{B^T B}$ for a square matrix $B$. The function $\tilde d$ will be a convenient modified distance for the proof. For convenience in notation, throughout the proof we use the shorthand $\lambdaMax := \lambdaMax(D^2 f(x^*))$ and $\lambdaMin := \lambdaMin(D^2 f(x^*))$.

Note that for $a\geq 0$ we have the following relationships
\begin{align} \label{eq:inclusion-ineq1}
 \|x\| \leq \frac{a}{\sqrt{\lambdaMax}} & \implies \tilde d(x) \leq a\\
\label{eq:inclusion-ineq2} \tilde d(x) \leq a & \implies \|x\| \leq \frac{a}{\sqrt{\lambdaMin}}.
\end{align}

By Taylor's theorem and the non-degeneracy of $x^*$, for any $C_2 > \frac{1}{2}$ there exists a neighborhood of $0$ such that
\begin{equation} \label{inequality2}
|f(x) - f(0)| \leq  C_1\tilde d(x)^2.
\end{equation}
Using the chain rule we see that along the path $\vx(t)$, the potential changes as
$$
\frac{d}{dt} f(\vx(t)) = -\|\grad f(\vx(t))\|.
$$
Let $C_2< 1$ be arbitrary. Again using Taylor's theorem and the non-degeneracy of $H$, for $\vx(t)$ in a neighborhood of $0$ we have that
\begin{align}
\nonumber \|\grad f(\vx(t))\| &\geq C_2 \|  H \vx(t)\| \\
\nonumber  &= C_2\| |H|^{1/2} |H|^{1/2} \vx(t)\| \\
\nonumber  &\geq C_2\sqrt{\lambdaMin}\||H|^{1/2} \vx(t)\| \\
  &= C_2\sqrt{\lambdaMin} \tilde d(\vx(t)), \label{eqn:taylor-estimate2}
\end{align}
where $\lambdaMin$ denotes the magnitude of the smallest-magnitude eigenvalue of $H$.
In turn
\begin{equation}\label{inequality1}
  -\frac{d}{dt} f(\vx(t)) \geq C_2 \sqrt{\lambdaMin} \tilde d(\vx(t)).
\end{equation}
Let $\hat r>0$ be such that the estimates \eqref{inequality2} and \eqref{eqn:taylor-estimate2} hold inside the closed ball $B_{\hat r}(0)$.
Suppose that $\vx(t)\in B_{\hat r}(0)$ for $t\in[t_1,t_2]$. Letting $e(t) := \tilde d(\vx(t))$ and integrating \eqref{inequality1} gives
$$
f(\vx(t_1)) - f(\vx(t_2)) \geq C_2\sqrt{\lambdaMin} \int_{t_1}^{t_2} e(s)ds.
$$
Let $r := \kappa^{-\frac{1}{2}}\hat r$. Suppose $\eta \leq \sqrt{\lambdaMax} r$ and note that by \eqref{eq:inclusion-ineq2}, $\tilde d(x) \leq \eta$ implies that $x\in B_{\hat r}(0)$.
Furthermore, suppose $e(t) \leq \eta$ for some $t\geq 0$, and
let $t_0$ be the first time where $e(t) \leq \eta$. Let $t_3$ be the last time when $e(t) = \eta$;
i.e., $t_3 = \sup\{t\in [0,\infty):~ e(t) \leq \eta\}$. If $t_3 = \infty$, then in an abuse of notation we let $f(\vx(\infty)) = \lim_{t\to\infty} f(\vx(t))$, where we note that the limit exists since $f(\vx(t))$ is monotone non-increasing in $t$.
It follows that
\begin{align}
  f(\vx(t_0)) - f(\vx(t_3)) & = \int_{t_0}^{t_3} -\frac{d}{ds} f(\vx(s))\,ds \\
  & \geq \int_{e(s) \leq \eta} -\frac{d}{ds} f(\vx(s))\,ds\\
  & \geq C_2 \sqrt{\lambdaMin}\int_{e(s) \leq \eta} e(s) \,ds,
\end{align}
where we use the fact that $\frac{d}{dt} f(\vx(t)) \leq 0$, and the previous inequality on subintervals where $e(\cdot) \leq \eta$.
Adding and subtracting $f(0)$ to the left hand side above and using \eqref{inequality2} we obtain
$$
\frac{2C_1}{C_2\sqrt{\lambdaMin}} \eta^2 \geq \int_{e(s) \leq \eta} e(s)ds.
$$
Markov's inequality \cite{federer2014geometric} then gives
\begin{align}
\calL^1\left(\{s: \eta \geq e(s) \geq \frac{\eta}{2} \}\right) & \leq  \frac{2}{\eta} \int_{e(s) \leq \eta} e(s)ds \\
&\leq \frac{4C_1}{\eta C_2\sqrt{\lambdaMin}} \eta^2\\
& = \frac{4C_1}{C_2\sqrt{\lambdaMin}}\eta.
\end{align}
We can iteratively apply this inequality to obtain
\begin{align}
  & \calL^1\left(\{s: \eta \geq  e(s) >0 \}\right)\\
  & = \sum_{i=0}^\infty \calL^1\left(\{s: \frac{\eta}{2^i} \geq e(s) \geq \frac{\eta}{2^{i+1}} \} \right) \\
  & \leq \sum_{i=0}^{\infty} \frac{4 C_2\eta}{C_2\sqrt{\lambdaMin}2^{i}}\\
  & \leq \frac{8 C_1}{C_2\sqrt{ \lambdaMin}}\eta. \label{eqn:time-bound}
\end{align}
By \eqref{eq:inclusion-ineq1} we see that
$
\{s:0<\|\vx(s)\|\leq r\} \subset \{s: 0<\tilde d(\vx(s)) \leq \sqrt{\lambdaMax}r\}.
$
Letting $\eta = \sqrt{\lambdaMax} r$ in \eqref{eqn:time-bound}, and letting $C:= \frac{8C_1}{C_2}$, we get
\begin{align}
  & \calL^1\Big(\{s: 0< \|\vx(s)\| \leq r \} \Big)\\
  & \leq \calL^1\left(\{s: 0< \tilde d(\vx(s)) \leq \sqrt{\lambdaMax}r \}\right) \leq C\frac{\sqrt{\lambdaMax}}{\sqrt{\lambdaMin}}r,
\end{align}
where we recall that $r = \kappa^{-1/2}\hat r$ and $\hat r$ is the radius of the ball where \eqref{inequality2} and \eqref{eqn:taylor-estimate2} hold and is dependent on $C_1$ and $C_2$. Since $C_1>\frac{1}{2}$ and $C_2 < 1$ were arbitrary, the constant $C$ may be brought arbitrarily close to $4$ with the range of permissible values of $r$ changing accordingly with the choice of $C_1$ and $C_2$. This proves the desired result.
\end{proof}

\vspace{-10pt}


\begin{proof}[Proof of Corollary \ref{cor-finite-time}]
First, we claim that critical points must be separated by a distance of at least $2r$. Let $x^*$ be a critical point. Then
\begin{align*}
  &\nabla f(x) = \int_0^1 D^2 f( (1-s)x^* + s x) (x-x^*) \,ds\\
  &= \int_0^{1} D^2 f(x^*)(x-x^*)\ds\\
  &+ \int_0^1\int_0^s D^3_{x-x^*} f( (1-\tau)x^* + \tau x)(x-x^*) \,d\tau\ds,
\end{align*}
\noindent where by $D^3_{x-x^*}$ we mean the matrix representing the third derivative evaluated in the direction $x-x^*$. We can then bound
\begin{displaymath}
  |\nabla f(x)| \geq \lambdaMin \|x-x^*\| - \frac{\hat C}{2}\|x-x^*\|^2
\end{displaymath}
\noindent where $\hat C$ is the bound on our third derivatives. Note that by Assumption \ref{a:strict-saddle} we have $2r \leq \frac{2\lambdaMin}{\hat C}$.  Thus we see that for any $x\in B_{2r}(x^*) \subset B_{\frac{2\lambdaMin}{\hat C}}(x^*)$ we have $\nabla f(x) \neq 0$. Hence critical points must be separated by a distance of at least $2r$.	

Now, let $\vx(t)$ be a classical solution of \eqref{eqn:normalized-dynamics} (which, by Theorem \ref{prop:stable-manifold}, holds for a.e. solution of \eqref{eqn:normalized-dynamics}). Let $[t_1,t_2] = I$ be the maximal interval of existence for this classical solution. Our goal is to prove that $(t_2-t_1)$ can be bounded uniformly.

To this end, we divide $I$ into two subsets, $I_c,I_0$, where $I_c$ are the times where $\|\vx(t) - x^*\|\leq r$ for some critical point $x^*$, and $I_0$ are points where $\|\nabla f(\vx(t))\| \geq \nu$.

Using the chain rule we see that $\frac{d}{dt} f(\vx(t)) = -\|\nabla f(\vx(t))\|$.
By Assumption \ref{a:invariant-set} and \eqref{def-M} we have $|f(\vx(t))| < M$ along any trajectory of \eqref{eqn:normalized-dynamics} starting in $B_R(0)$. Thus, we immediately have that $|I_0| < 2M\nu^{-1}$.

Let $\kappa = \frac{\lambdaMax}{\lambdaMin}$.
By Proposition \ref{prop:time-bound} we can spend at most time $C\sqrt{\kappa}r$ near any particular critical point. Since critical points are separated by at least distance $2r$, we can cover all the critical points with disjoint balls of radius $r$. By then estimating the volume, the total number of critical points within distance $R$ of the origin is at most $\frac{(R+r)^d}{r^d}$. Thus we find that $|I_c| < C\sqrt{\kappa}r \frac{(R+r)^d}{r^d}$.

In summary, we find that $|I| \leq 2M\nu^{-1} + C\sqrt{\kappa} \frac{(R+r)^d}{r^{d-1}}.$
This implies that classical trajectories can be of length at most $M\nu^{-1} + C\sqrt{\kappa}\frac{(R+r)^d}{r^{d-1}}$. Since a.e. initial condition does not reach any saddle point, almost every initial condition will converge to a local minimizer of $f$ in $2M\nu^{-1} + C\sqrt{\kappa}\frac{(R+r)^d}{r^{d-1}}$ time. This concludes the proof.
\end{proof}

\bibliographystyle{IEEEtran}
\bibliography{myRefs}

\end{document}